%% file: chain_bounding_arxiv.tex
\documentclass{amsart}
\usepackage[utf8]{inputenc}
\input{header-bounding}

\input{lean_snippets}
\usepackage[numbers]{natbib}

\newtheorem{theorem}{Theorem}
\newtheorem{proposition}[theorem]{Proposition}
\newtheorem{corollary}[theorem]{Corollary}
\newtheorem{lemma}[theorem]{Lemma}
\newtheorem*{lemma*}{Lemma}

\theoremstyle{definition}
\newtheorem{definition}[theorem]{Definition}

\hypersetup{
  pdftitle={Chain Bounding, the leanest proof of Zorn's lemma, and an illustration of computerized proof formalization},
  pdfauthor={G.L. Incatasciato, P. Sánchez Terraf},
  colorlinks,
  breaklinks,
  linkcolor={blue!35!black},
  citecolor={blue!35!black},
  urlcolor={blue!35!black}
}
\begin{document}

\title[Chain bounding, Zorn's lemma, and proof formalization]{Chain bounding, the leanest proof of Zorn's lemma, and an illustration of
  computerized proof formalization}
\date{}
\author[G.L. Incatasciato]{Guillermo L. Incatasciato}
\email{guillermo.incatasciato@mi.unc.edu.ar}

\author[P. Sánchez Terraf]{Pedro Sánchez Terraf}
\email{psterraf@unc.edu.ar}
\thanks{Universidad Nacional de Córdoba.  Facultad de Matemática, Astronomía,  Física y
  Computación.
  \\
  Centro de Investigación y Estudios de Matemática (CIEM-FaMAF),
  Conicet. Córdoba. Argentina.\\
  Supported by Secyt-UNC project 33620230100751CB and Conicet PIP project 11220210100508CO}

\begin{abstract}
  We present an exposition of the \emph{Chain Bounding Lemma}, which is a common
  generalization of both Zorn's Lemma and the Bourbaki-Witt fixed point theorem.
  The proofs of these results through the use of Chain Bounding are amongst the
  simplest ones that we are aware of. As a by-product, we show that for every
  poset $P$ and function $f$ from the powerset of $P$ into $P$, there exists a
  maximal well-ordered chain whose family of initial segments is appropriately closed
  under $f$.

  We also provide an introduction to the process of “computer formalization” of
  mathematical proofs by using \emph{proofs assistants}. As an illustration, we
  verify our main results with the Lean proof assistant.
\end{abstract}

\maketitle

\section{Introduction}
\label{sec:introduction}

This paper grew out of the search for an elementary proof of Zorn's Lemma. One
such proof was obtained by the first author, %
which is similar to the one by Lewin \cite{lewin}.%

After a careful examination, the authors realized that the method of proof
actually yielded a pair of new, similar principles: \emph{Chain Bounding} and the
\emph{Unbounded Chain Lemma}, which state the impossibility of finding strict
upper bounds of linearly ordered subsets of posets. The first one is more
fundamental, since it does not depend on the Axiom of Choice, which
states:
\begin{quote}
  ($\AC$) For every family $\{ A_i \mid i\in I \}$ of nonempty sets, there exists a
  function $f:I \to \bigcup_{i\in I} A_i$ such that for all $i\in I$, $f(i)\in A_i$.
\end{quote}
When Chain Bounding is coupled 
with $\AC$, it implies the second principle and then Zorn's Lemma. Chain Bounding
also implies the Bourbaki-Witt fixed point theorem; all these results are
in Section~\ref{sec:chain-bound-appl}. 

Our original proof of Chain Bounding proceeded by contradiction, where a few
relevant concepts were defined; this proof is essentially the one that
appears in Appendix~\ref{sec:zorns-its-own}, where it is used to show Zorn's
Lemma in a self-contained manner. We realized that it was better instead to
present those concepts (“good chains” and their comparability) independently to
obtain positive results. These appear in
Section~\ref{sec:greatest-good-chain}, and Chain Bounding is now proved as a
consequence of its main Theorem~\ref{th:greatest-good-chain}, the existence of a
greatest good chain.
Nevertheless, the main advantage of Chain Bounding in comparison to
Theorem~\ref{th:greatest-good-chain} is its straightforward statement and
consequences, which make it more appealing as a “quotable principle”.

As a way to discuss the correctness and level of detail of our
arguments, we introduce the subject of “computer formalization” of mathematics
in Section~\ref{sec:comp-form-math}, and in Appendix~\ref{sec:comp-form-lean} we
present a brief description of the verification of our main results, using the
\emph{Lean} proof assistant.

\textit{A word on our intended audience.} As one of the reviewers indicated,
different parts of this paper can be better suited for varied kinds of
readers. In general, the paper should be accessible to mature undergraduates,
but the main focus changes a bit across sections. Instructors can benefit from
having the succinct proof of Zorn's Lemma from Appendix~\ref{sec:zorns-its-own},
or directly adopting Chain Bounding as a tool if they are teaching a course on
posets. The discussion from Corollary~\ref{cor:unbounded-implies-AC} to the end
of Section~\ref{sec:chain-bound-appl} goes a little deeper into the details and
assume a bit more of set-theoretic background. Finally, we hope that the general
exposition of proof assistants be accessible to a wider public; the detailed
example in Appendix~\ref{sec:comp-form-lean} is a bit more challenging but might
be attractive to younger students which are more familiar with computer
technology.

\section{The greatest good chain}
\label{sec:greatest-good-chain}

We introduce some notation. Let $P$ be a poset and $C\subseteq P$; we say that
$s\in P$ is a \emph{strict upper bound} of $C$  if $\forall c\in
C,\ c<s$.  Furthermore, if $S\sbq C$, we say that $S$ is an \emph{initial
segment} of $C$ (“$S\segm C$”) if for all $x\in C$, $x\leq y \in S$ implies
$x\in S$. We will usually omit the word “initial” and simply say “$S$ is a
segment of $C$”. The strict version of the segment relation is denoted by
$S\propsegm C$, that is, $S\segm C$ and $S\neq C$.
Finally, $\Pow(X)$ denotes the powerset of the set  $X$.

\begin{definition}
  Let $g: \Pow(P) \to \Pow(P)$ be given.
  We say that a chain $C\subseteq P$ is \emph{good for $g$} if for all  $S\propsegm C$,
  $S\propsegm g(S) \segm C$.
\end{definition}
When understood from the context, we omit “for $g$”.
We have the following key result.
\begin{lemma}[Comparability]\label{lem:comparability}
  Let $P$ be a poset and $g: \Pow(P) \to \Pow(P)$. If $C_1$ and $C_2$ are
  good chains, one is a segment of the other.
\end{lemma}
\begin{proof}
  Let $\calS$ be the family of mutual
  segments of both $C_1$ and $C_2$. Hence $\union \calS$ is also a
  mutual segment. If $\union \calS$ is different from both $C_1$
  and $C_2$, then $g(\union  \calS)$ should be a
  mutual segment since both are good; but this contradicts the fact that $g(\union
  \calS)\not\subseteq\union \calS$.
\end{proof}

\begin{lemma}\label{lem:union-good-chains}
  The union of a family $\calF$ of good chains is a good chain.
\end{lemma}
\begin{proof}
  The union $U\defi \union \calF$ is a chain by Comparability.
  
  Note that every good chain $D$ such that $D\subseteq U$ is a segment of $U$:
  Suppose that $c\in U$ and $c<d\in D$. Then $c\in C$ for
  some good $C\in\calF$. If $C$ is a segment of $D$, we have $c\in D$ and are done. Otherwise, the
  converse relation holds by Comparability and then we also have
  $c\in D$.

  We will see that $U$ is good. Let $S\subsetneq U$ be a proper segment
  of $U$. Then, there exists $d\in U$ such that $\forall
  c\in S,\; c<d$. Let $D$ be a good chain such that $d\in D$.
  Since $D$ is a segment of $U$, all those $c$ belong to $D$.
  We conclude $S\subseteq D$, and since  $S$ is a
  segment of $U$, it is a segment of $D$ and it is proper because $d\in
  D \sm S$.  Then $g(S)$ is segment of $D$, and hence
  $g(S)$ is a segment of $U$.  
\end{proof}

By considering the union of \emph{all} good chains, we readily obtain:
\begin{theorem}[Greatest Good Chain]\label{th:greatest-good-chain}
  Let $P$ be a poset and $g: \Pow(P) \to \Pow(P)$. The family of all good
  chains has a maximum under inclusion.\qed
\end{theorem}

\section{Chain Bounding and applications}
\label{sec:chain-bound-appl}

The following lemma is the key to all of what follows. In some
sense, it might be regarded as a non-$\AC$ version of Zorn's lemma.
\begin{lemma}[Chain Bounding]\label{lem:chain-bounding}
  Let $P$ be a poset. There is no assignment of a strict upper bound to each
  chain in $P$.
\end{lemma}
\begin{proof}
  Assume, by way of contradiction, that $f(C)$ is
  a strict upper bound of $C$ for each chain $C\sbq P$. Hence $C$ is a proper
  segment of $g(C) \defi C\cup  \{f(C)\}$; extend this $g$ arbitrarily to the
  rest of the  subsets of $P$. By
  Theorem~\ref{th:greatest-good-chain}, there exists a greatest good chain
  $U$ for $g$. But this is a contradiction, since $g(U)$ is easily
  seen to be a good chain, but $g(U)\not\subseteq U$.
\end{proof}
The wording of the Chain Bounding Lemma is a bit awkward, since it is actually a
negation. However, if we now invoke $\AC$, we get the following
more natural statement.
\begin{lemma}[Unbounded Chain]\label{lem:unbounded-chain}
  Assume $\AC$. For every poset $P$ there exists a chain $C\subseteq P$ with no strict upper bound.
\end{lemma}
\begin{proof}
  By way of contradiction, assume that for every chain $C \subseteq P$ there
  exists a strict upper bound. Using $\AC$, let $f$ assign to each
  $C$ such a bound. But this contradicts Chain Bounding.
\end{proof}

We now turn to applications. The first one is very simple
proof of
Zorn's Lemma (obviously taking into account the lemmas proved so far).
\begin{corollary}[Zorn]
  If a  poset $P$ contains an upper bound for each chain, it has a maximal element.
\end{corollary}
\begin{proof}
  By the Unbounded Chain Lemma, take $C\subseteq P$ without strict upper
  bounds. Then any upper bound of $C$ must be maximal in $P$.
\end{proof}
A direct, self-contained proof of Zorn's Lemma condensing all the ideas
discussed up to this point appears in Appendix~\ref{sec:zorns-its-own}
below. This includes some simplifications that also apply to a direct proof of
Chain Bounding; for instance, the definition of good chain is a bit shorter and one
only needs to show that the union of all good chains is good.

Since Zorn's Lemma implies $\AC$, we immediately have:
\begin{corollary}\label{cor:unbounded-implies-AC}
  The Unbounded Chain Lemma is equivalent to $\AC$ over Zermelo-Fraenkel set
  theory.\qed
\end{corollary}

Kunen points out in \cite{kunen2011set} that, for those not familiar with set
theory, it may not be clear why Zorn's Lemma should be true, since the
best-known proofs make use of ordinals and transfinite recursion. 
Our highest hope is that after seeing our proof, the old joke turns
into \emph{“$\AC$ is obviously true, the Well-Ordering Theorem is obviously
false, and Zorn's Lemma\dots holds by Chain Bounding!”}.

Lewin, in \cite{lewin}, provides a very short proof without the need for
ordinals or recursion, but making use of well-ordered chains.
In \cite{lang}, Lang
presents a proof using the Bourbaki-Witt fixed point theorem of order theory,
without even employing the concept of well-ordered
set. Finally, Brown \cite{brown} gives a beautiful and simple proof inspired
by Lang's but without the need for Bourbaki-Witt. This proof is slightly
indirect since it actually proves the Hausdorff Maximal Principle and considers
“closed” subsets in the poset of chains of the original poset ordered by
inclusion. (Other proofs in the same spirit can be found in Halmos
\cite{halmos1960naive} and Rudin \cite{Rudin}.)

The proof we presented here is not as short as Lewin's but it is more
elementary since there is no use of (the
basic theory of) well-orders in an explicit way. The main
difference in method that allows us to avoid them is the generalization of
his definition of “conforming chains” by considering general initial
segments instead of \emph{principal} ones (i.e., of the form $\{ x\in P \mid x\leq p \}$ for some
$p\in P$). This move allows us to use the stronger expressiveness achieved by
talking about general segments (indirectly referring to the powerset of $P$), but
avoids referring to “second order” chains (i.e., chains in the poset of
chains).

In spite of this simplification, the fundamental character of the
concepts of well-order and well-foundedness in general should be strongly emphasized. These are unavoidable in
a sense; actually good chains for functions that add at most one element (such
as the one in the proof of Chain Bounding)  are well-ordered:
\begin{proposition}\label{prop:good-well-ordered}
  Let $(P,<)$ be poset, $f:\Pow(P) \to P$ and let $g(C) \defi C\cup\{ f(C)
  \}$. Every good chain for $g$ is  well-ordered by $<$.
\end{proposition}
\begin{proof}
  We leave to the reader the verification of the fact, under the
  assumptions, that every segment $D$ of a good chain is good.
  
  Let $C$ be a good chain and assume $X\subseteq C$ is nonempty. Let $S$ be the
  set of strict \emph{lower} bounds of $X$ in $C$. Since $X\neq\varnothing$, $S$
  is a proper segment of $C$, and goodness ensures that $S\neq g(S)=S\cup\{ f(S) \}$
  is also a segment of $C$. Since $f(S)\notin S$, there is some $x\in X$ such
  that $x\leq f(S)$. We claim that any such $x$ must be equal to $f(S)$ and
  hence it is the minimum element of $X$. For this, consider
  the good subchain $D\defi \{ c\in C \mid c\leq x \}$ of $C$. Since $S$ is
  likewise a proper segment of $D$, $f(S)\in g(S) \subseteq D$ and hence we obtain the claim.  
\end{proof}
Moreover, it can be shown that a chain of a poset $P$ is well-ordered if and
only it is a good chain for some $g$ as above.

Our second application is the aforementioned fixed point theorem.
\begin{corollary}[Bourbaki-Witt]
  Let $P$ be a non-empty poset such every chain $C\subseteq P$ has a least upper
  bound. If $h : P\rightarrow P$ satisfies $x \leq h(x)$ for all $x\in P $, then
  $h$ has a fixed point, i.e., there is some $x\in P$ such that $x=h(x)$.
\end{corollary}
\begin{proof}
  Assume by way of contradiction that $x<h(x)$ for all $x\in P$. But then $f(C)
  := h(\sup C)$ immediately contradicts Chain Bounding.
\end{proof}
Note that the greatest good chain for $C
\stackrel{g}{\longmapsto} C \cup \{ h(\sup C) \}$ is the least complete subposet
of $P$ closed under $h$, and its ordinal length is (the successor of) the number of iterations of
$h$ needed to reach its least fixed point starting from the bottom element of
$P$.

It is relevant here that Chain Bounding does not depend on $\AC$, since
Bourbaki-Witt can be proved without using it. This is another reason why we
consider Chain Bounding the central item of this work.

\input{intro_formalization}

\appendix

\section{Zorn's on its own}
\label{sec:zorns-its-own}

For ease of reference, we streamline the arguments above to obtain a compact
version of the proof.

Recall that for a poset $P$,
$s\in P$ is a \emph{strict upper bound} of $C\sbq P$ if $\forall c\in C,\ c<s$; and
that $S\sbq C$ is a \emph{segment} of $C$ if for all
$x\in C$, $x\leq y \in S$ implies $x\in S$.
\begin{lemma*}[Zorn]
  If a  poset $P$ contains an upper bound for each chain, it has a maximal element.
\end{lemma*}
\begin{proof}
  Assume by way of contradiction that $(P,\leq)$ does not have a maximal
  element. Hence, for every chain $C \subseteq P$ there exists a strict
  upper bound (otherwise, any upper bound of $C$ would be
  maximal). Using the Axiom of Choice, let $g$ assign $C\cup\{s\}$ to each chain $C\subseteq
  P$, where $s$ is any such bound.

  A chain $C\sbq P$ is deemed to be \emph{good} whenever
  \begin{quote}
    (*) If $S\neq C$  is a segment of $C$, then
    $g(S)$ also is.
  \end{quote}
  We need the following property of good chains:
  \begin{quote}
    (\emph{Comparability}) If $C_1, C_2$ are good, one is a
    segment of the other.
  \end{quote}
  To prove it,
  let $\calS$ be the family of mutual
  segments of both $C_1$ and $C_2$. Hence $\union \calS$ is also a
  mutual segment. If $\union \calS$ is different from both $C_1$
  and $C_2$, then $g(\union  \calS)$ should be a
  mutual segment by (*); but this contradicts the fact that $g(\union
  \calS)\not\subseteq\union \calS$.

  Let $U$ be the union of all good chains, which is a chain by Comparability.

  Note that every good chain $D$ is a segment of $U$: Suppose that $c\in U$
  and $c<d\in D$. Then $c\in C$ for some good $C$. If $C$ were not a segment of
  $D$, then the converse relation would hold by Comparability and
  then we would also have $c\in D$.

  We will see that $U$ is good. Let $S\subsetneq U$ be a proper segment
  of $U$. Then, there exists $d\in U$ such that $\forall
  c\in S,\; c<d$. Let $D$ be a good chain such that $d\in D$.
  Since $D$ is a segment of $U$, all those $c$ belong to $D$.
  We conclude $S\sbq D$, and since  $S$ is a
  segment of $U$, it is a segment of $D$ and it is proper because $d\in
  D \sm S$.  Then $g(S)$ is segment of $D$, and hence
  $g(S)$ is a segment of $U$.

  We reach a contradiction, since $g(U)$ is also a good
  chain, but $g(U)\not\subseteq U$.
\end{proof}

\input{formalization}

\paragraph{Acknowledgment} We are very grateful to both reviewers and to the
Editorial Board for their comments that helped to enhance our
presentation. Particularly, we thank Reviewer 3 for an extremely dedicated
reading, and Reviewer 2 for insightful comments on the intended audience of the
paper.

\input{chain_bounding.bbl}

\end{document}

%% file: header-bounding.tex
\usepackage{amsmath,amsfonts,amssymb}
\usepackage{rsfso} %
\usepackage{bbm}  %
\usepackage{tikz}
\usepackage[normalem]{ulem}
\usepackage{multidef}
\usepackage{verbatim}
\usepackage{stmaryrd} %
\usepackage{hyperref}
\usepackage{xcolor}

\newcommand{\modelo}[1]{\mathbf{#1}}
\newcommand{\axiomas}[1]{\mathit{#1}}
\newcommand{\clase}[1]{\mathsf{#1}}

\multidef{\clase{#1}}{Card,HC,HF,Lim,On->Ord,Reg,WF,Ord,WO,LO,LTS}

\DeclareMathAlphabet{\mathbbm}{U}{bbm}{m}{n}

\multidef[prefix=cal]{\mathcal{#1}}{A-Z}

\multidef{\modelo{#1}}{BB->B,CC->C,NN->N,QQ->Q,RR->R,ZZ->Z}

\multidef[prefix=p]{\mathbb{#1}}{A-Z}

\newcommand{\Pow}{\mathop{\calP}}

\newcommand{\R}{\mathbb{R}}
\newcommand{\N}{\mathbb{N}}

\renewcommand{\phi}{\varphi}

\newcommand{\defi}{\mathrel{\mathop:}=}

\newcommand{\AC}{\axiomas{AC}}

\newcommand{\union}{\mathop{\textstyle\bigcup}}
\newcommand{\sm}{\smallsetminus}
\newcommand{\sbq}{\subseteq}

\renewcommand{\leq}{\leqslant}

\newcommand{\segm}{\sqsubseteq}
\newcommand{\propsegm}{\sqsubset}

\makeatletter
\@ifpackageloaded{txfonts}\@tempswafalse\@tempswatrue
\if@tempswa
  \DeclareFontFamily{U}{txsymbols}{}
  \DeclareFontFamily{U}{txAMSb}{}
  \DeclareSymbolFont{txsymbols}{OMS}{txsy}{m}{n}
  \SetSymbolFont{txsymbols}{bold}{OMS}{txsy}{bx}{n}
  \DeclareFontSubstitution{OMS}{txsy}{m}{n}
  \DeclareSymbolFont{txAMSb}{U}{txsyb}{m}{n}
  \SetSymbolFont{txAMSb}{bold}{U}{txsyb}{bx}{n}
  \DeclareFontSubstitution{U}{txsyb}{m}{n}
  \DeclareMathSymbol{\aleph}{\mathord}{txsymbols}{64}
  \DeclareMathSymbol{\beth}{\mathord}{txAMSb}{105}
  \DeclareMathSymbol{\gimel}{\mathord}{txAMSb}{106}
  \DeclareMathSymbol{\daleth}{\mathord}{txAMSb}{107}
\fi
\makeatother

%% file: lean_snippets.tex
\usepackage{newunicodechar,listings,marvosym}
\definecolor{keywordcolor}{rgb}{0.7, 0.1, 0.1}   %
\definecolor{tacticcolor}{rgb}{0.0, 0.1, 0.3}    %
\definecolor{commentcolor}{rgb}{0.4, 0.4, 0.4}   %
\definecolor{stringcolor}{rgb}{0.5, 0.3, 0.2}    %
\definecolor{symbolcolor}{rgb}{0.1, 0.2, 0.7}    %
\definecolor{sortcolor}{rgb}{0.1, 0.5, 0.1}      %
\definecolor{attributecolor}{rgb}{0.7, 0.1, 0.1} %
\definecolor{errorcolor}{rgb}{1, 0, 0}           %

\newunicodechar{α}{\ensuremath{\alpha}}
\newunicodechar{₀}{\ensuremath{{}_0}}
\newunicodechar{∀}{\ensuremath{\forall}}
\newunicodechar{∃}{\ensuremath{\exists}}
\newunicodechar{∈}{\ensuremath{\in}}
\newunicodechar{∉}{\ensuremath{\notni}}
\newunicodechar{∧}{\ensuremath{\wedge}}
\newunicodechar{∨}{\ensuremath{\vee}}
\newunicodechar{∪}{\ensuremath{\cup}}
\newunicodechar{≠}{\ensuremath{\neq}}
\newunicodechar{≤}{\ensuremath{\leq}}
\newunicodechar{⊂}{\ensuremath{\subset}}
\newunicodechar{⊆}{\ensuremath{\subseteq}}
\newunicodechar{⋃}{\ensuremath{\bigcup}}
\newunicodechar{⟨}{\ensuremath{\langle}}
\newunicodechar{⟩}{\ensuremath{\rangle}}
\newunicodechar{⊏}{\ensuremath{\sqsubset}}
\newunicodechar{⊑}{\ensuremath{\sqsubseteq}}
\newunicodechar{✝}{\Cross}
\newunicodechar{⊢}{\ensuremath{\vdash}}
\newunicodechar{⦃}{\ensuremath{\{\![}}
\newunicodechar{⦄}{\ensuremath{]\!\}}}
\newunicodechar{₁}{\ensuremath{{}_1}}
\newunicodechar{₂}{\ensuremath{{}_2}}

%% file: intro_formalization.tex
\section{Computer formalization of mathematics}
\label{sec:comp-form-math}

Although Lewin's proof is even shorter than ours, it could be debatable whether
it is more straightforward for a general audience. Specifically, some
acquaintance with the basic theory of well-orders seems to be required to
understand it in full; in other words, not all the details of the proof are
completely disclosed (note, e.g. the Mathematics Stack Exchange questions
\cite{stackexchange_lewin} and \cite{stackexchange_lewin2} and their answers).

Our combo Chain-Bounding/Zorn is not immune to the same criticism: Some tricky
details could still have been swept under the rug---or, even worse, some
mistake. The question presents itself: How could one assert that all details of
a proof have been taken care of, and how to be sure that absolutely all steps
are correct? This concern will certainly connect in the reader's mind to the
“rigorization” process in Mathematics during the late XIX century and perhaps to
the foundational aspects of mathematical logic. It might also raise an alert of
sorts (recalling Russell and Whitehead's \emph{Principia}) of the infeasibility
of the task.

The fact of the matter is that thanks to a steady development in Computer
Science, modern programming technologies called \emph{proof assistants} (PAs)
have been created, that allow one to write a mathematical proof in a formal language
and the computer can then verify it for correctness. This process is called
“formalization”, “(formal) verification”, or “mechanization”.

\subsection{Use cases}
PAs not only make it possible to trace every single detail of simple results
such as the ones in this paper; highly non trivial and recent research-level
mathematical results have been checked in diverse PAs. Notable formalizations include the
Four Color Theorem \cite{10.1007/978-3-540-87827-8_28}, the Odd-Order Theorem
\cite{10.1007/978-3-642-39634-2_14} (both using the Coq PA), and the proof of the Kepler
Conjecture \cite{MR3659768} (using a combination of HOL-Light and Isabelle/HOL
PAs). Even more recently, Gowers, Green, Manners, and Tao settled the Polynomial Freiman-Ruzsa
conjecture \cite{2023arXiv231105762G}, sharing their results on
November 9th, 2023; Tao started a project \cite{digitisation-pfr} to formalize them (including many
preliminaries) using the \emph{Lean} PA \cite{DBLP:conf/cade/Moura021}
and it was finished \cite{pfr-anounce} less than a month
after the manuscript was released! The process was described by Tao in his blog
\cite{formalizing-pfr}.

There are also many ongoing attempts to use PAs in teaching. A regular workshop
\cite{Narboux2024} discusses the potential uses of the latter and other computer
technologies in education, and there are actually many courses that already use
PAs. Examples are Massot's calculus course \cite{massot-course} for the second semester of BSc,
and Macbeth's Math~2001 at Fordham University, supported by an excellent
introductory material \cite{macbeth-course} to Lean. More pointers to the use of
Lean in education can be found in Avigad's talk
\cite{formal-assistants-edu-avigad} at the Fields Institute.

\subsection{A bird's eye comparison}
To this day, there are plenty of fine pieces describing the different kinds PAs
available, and their use, benefits and pitfalls. To name a few, Harrison, Urban,
and Wiedijk recall the rich history of PAs in their
\cite{DBLP:series/hhl/HarrisonUW14}, while Koutsoukou-Argyraki shares her
experiences using Isabelle/HOL in \cite{koutsoukou-argyraki_2020} and, after
presenting an insightful summary, interviews some of the relevant figures
involved with this technology in her \cite{angeliki}. Even a recent article in Nature
magazine discusses PAs \cite{nature-lte}.
Having all these resources at hand, we will only briefly overview some details
of current PAs.

PAs come in a great variety, both internally/foundationally speaking and from
the user point of view. Regarding foundations, most of them differ from the
traditional set-theoretic ones, being based on several kinds of \emph{type
theories}. For the purpose of this paper, we can take “type” and “set” to be
synonyms, but the main difference is that every mathematical object in
consideration belongs to \emph{exactly one} type; as an example of this
stipulation, a natural number $n$ is not to be considered a real one, but
instead there is an appropriate embedding $\N\hookrightarrow \R$ which sends $n$
to the corresponding real. Isabelle/HOL and HOL-Light are based on the “simple”
theory of types; the axiomatics of the latter PA is surprisingly terse, and a
description can be found at Hales \cite{MR2463990}. On the other hand, Coq and
Lean support “dependent type theory”, which is stronger.%
\footnote{%
  There have been interesting debates (e.g., Buzzard's challenge to the Isabelle
  community \cite{buzzard-icm-2022}) on what is the actual impact of the
  differing strengths of PAs' foundations concerning the possibility to
  formalize research-level mathematics. We only point out that HOL (simple type
  theory) systems are weaker than Zermelo-Fraenkel set theory with Choice, which
  in turn is weaker than the dependent type theory with universes used in
  Lean. We are also not discussing the \emph{constructivist} aspect of
  foundations, which may enable or prevent the extraction of algorithms from the
  formalized proofs.
}

There are notable differences also in the language used by PA to write proofs
and the overall user interfaces. Some PAs are actually full-fledged programming
languages that are also suited to formalized proofs (e.g., Agda, Lean), while
some others may produce computer programs from a proof (Coq,
Isabelle/HOL). Concerning the way proofs are written, two major flavors are
available \cite[Sect.~6.2]{DBLP:series/hhl/HarrisonUW14}:
\begin{itemize}
\item \emph{declarative} proofs state intermediate facts that lead to the
  expected result (“We have $X$, hence $Y$, and then $Z$”); and
\item \emph{procedural} ones, which consist in a series of instructions to be
  performed to current “goal” or assertion to be proved (“Subtract $x$ from both
  sides, unfold $f$'s definition, conclude”).
\end{itemize}
In a very rough first approximation, procedural proofs are easier to write (and
more useful for exploration) and declarative proofs are easier to read (in some
cases, even without the need to use the corresponding PA to examine them).

Most PAs support a combination of both methods, but there are some extreme
cases. For instance, the Naproche variant of Isabelle is completely declarative
in the sense that it tries to solve each intermediate step automatically; this
results in beautiful formal proofs \cite{10.1007/978-3-030-81097-9_2} but it is
somewhat limited in power. On the other end, proofs in Agda are constructive,
and what one actually does is to define a function that provides the
witness/result required by the statement.

Most of our experience comes from using both Isabelle and Lean. Isabelle allows
its users
to write \emph{proof documents} that can be read without the need of a computer;
but it should be emphasized that this requires extra effort, and people do not
always formalize things in this way. Nevertheless, the HOL variant of Isabelle has
a noteworthy component (\emph{Sledgehammer}) able to deal automatically with many of the intermediate
steps. Lean, on the other hand, uses a more procedural
way to work. After consulting on this point, it was observed that human-readable
versions of the proofs could be derived from the formal artifacts, as
exemplified during Massot's talk \cite{massot-map2023} at the 2023 IPAM Workshop
on \emph{Machine Assisted Proofs}. Tools like Sledgehammer are also being
actively developed for Lean.

A further difference between these two PAs might be the composition of the user
base. Thanks to several high-profile projects, many mathematicians have started to
use Lean during the recent years; and while there are mathematicians working
with Isabelle, we believe their community hosts a majority of computer
scientists.

In Appendix~\ref{sec:comp-form-lean}, we present the formalization of the main
results of the paper using the Lean PA, which may serve to showcase the look of
the code using it and some characteristics of its user interface.
We also invite the reader to visit the
website
\begin{center}
  \url{https://leanprover-community.github.io/learn.html}
\end{center}
where many indications can be found on how to start using Lean, and the online
forum
\begin{center}
  \url{https://leanprover.zulipchat.com/}
\end{center}
fosters the growing community of users
and is very welcoming to newcomers.

%% file: formalization.tex
\section{A computer formalization in Lean}
\label{sec:comp-form-lean}

In this section, we present the formalization of the main results of the paper
using the Lean proof assistant. It takes up 420 lines of
code and it is available at
\begin{center}
  \url{https://github.com/sterraf/ChainBounding},
\end{center}
For the remainder of this section, we will focus on explaining only a fraction of the details involved (and
in particular, some notations will not be dealt with); our main objective is to
make a point that it is possible to translate mathematical reasoning into the
computer, in a way that at least partially resembles the way it is done on paper. We hope that the
reader's curiosity will be sufficiently motivated in order to visit the
mentioned resources and to learn more about formalization and Lean. 

We
start our Lean file by \emph{importing} basic results on chains, and the
definition of complete partial orders (which appear in the Bourbaki-Witt
Theorem).

\begin{lstlisting}
import Mathlib.Order.Chain
import Mathlib.Order.CompletePartialOrder

variable {α : Type*}
\end{lstlisting}
The last line above indicates that we will be talking about a “type” $\alpha$ (which, in
the type theory of Lean roughly corresponds to a set, or perhaps more
appropriately, a set \emph{underlying} some structure). Greek letters are
commonly used for types, and here this $\alpha$ will replace our $P$ from above.

We highlight some of the basic definitions. For instance, “$S$ is a
(proper) segment of $C$” is defined in the following way:

\begin{lstlisting}
def IsSegment [LE α] (S C : Set α) : Prop := S ⊆ C ∧ ∀ c ∈ C, ∀ s ∈ S, c ≤ s → c ∈ S

def IsPropSegment [LE α] (S C : Set α) : Prop := IsSegment S C ∧ S ≠ C
\end{lstlisting}
The arguments $S$ and $C$ appear declared as belonging to the powerset of
$\alpha$, which in Lean is written as \lstinline{Set}~$\alpha$. The declaration in
square brackets is an \emph{implicit} argument stating that $\alpha$ belongs to
the class of types having the $\leq$ notation defined (which is the bare minimum
to be able to interpret the right hand side). After a few more lines, the
declaration
\lstinline{variable [PartialOrder α]} states that we will be assuming a partial
order structure on $\alpha$.

We can actually set up infix notation for \lstinline{IsSegment} and
\lstinline{IsPropSegment} in order to be able to write expressions as $S\segm C$
and $S \propsegm C$, as shown below.

After a new concept is introduced, a customary requisite is to write some
extremely basic lemmas which allow one to work with it. These are referred to
by the name \emph{API}, an acronym for  “application
programming interface”, a concept that comes from Computer Science. In our formalization,
part of the API comprises all the possible transitivity lemmas involving
$\segm$, $\propsegm$, or both. 

We describe the formalization of the fact, used at the beginning of the proof of
Lemma~\ref{lem:comparability}, that the union of a family of segments is a segment. The
formalized statement is the following (where $\union_0$ denotes the operator of
union of a family), and the \lstinline{by}
keyword signals the start of the (tactic) proof:
\begin{lstlisting}
lemma sUnion_of_IsSegment {F : Set (Set α)} (hF : ∀M ∈ F, M ⊑ C) : ⋃₀ F ⊑ C := by
\end{lstlisting}
Since $\union_0 F \segm C$ is defined by a conjunction, its justification is
\emph{constructed} by providing proofs for each conjunct. Each of those proofs
appear indented and signaled by “$\cdot$” below. We will analyze the first sub-proof line by line. 
\begin{lstlisting}
  constructor
  · intro s sInUnionF
\end{lstlisting}
\newpage
\begin{lstlisting}
    obtain ⟨M, MinF, sinM⟩ := sInUnionF
    exact (hF M MinF).1 sinM
  · intro c cinC s sInUnionF cles
    obtain ⟨M, MinF, sinM⟩ := sInUnionF
    exact ⟨M, MinF, (hF M MinF).2 c cinC s sinM cles⟩
\end{lstlisting}
Right after writing
\lstinline{constructor} and the subsequent dot, the VS Code editor echoes:
\begin{lstlisting}
  α : Type u_1
  inst ✝ : PartialOrder α
  C : Set α
  F : Set (Set α)
  hF : ∀ M ∈ F, M ⊑ C
  ⊢ ⋃₀ F ⊆ C
\end{lstlisting}
\lstset{basicstyle=\ttfamily}
This “InfoView” lists all terms available to work (hypotheses are also included as
“\lstinline{Prop}\hspace{0.04em}ositional” terms), and the current \emph{goal}
(which, for this sub-proof, is the inclusion on the last line).
\lstset{basicstyle=\ttfamily\small}

The natural way of producing a proof of that inclusion (defined by
\lstinline{∀ ⦃s⦄, s ∈ ⋃₀ F → s ∈ C}), is to \emph{introduce} two new variables
named \lstinline{s} and \lstinline{sInUnionF},
\begin{lstlisting}
  · intro s sInUnionF
\end{lstlisting}
whose types (“\lstinline{α}” and
“\lstinline{s ∈ ⋃₀ F}”, respectively) are deduced from the previous goal. Now
the InfoView turns into
\begin{lstlisting}
  α : Type u_1
  inst ✝ : PartialOrder α
  C : Set α
  F : Set (Set α)
  hF : ∀ M ∈ F, M ⊑ C
  s : α
  sInUnionF : s ∈ ⋃₀ F
  ⊢ s ∈ C
\end{lstlisting}
From \lstinline{sInUnionF}, which states by definition that \lstinline{s} belongs to some
element of \lstinline{F}, we obtain such an element \lstinline{M} and further
terms/hypothesis that state the relations among them with the next line (where
the “\lstinline{:=}” can be read as “from”):
\begin{lstlisting}
   obtain ⟨M, MinF, sinM⟩ := sInUnionF
\end{lstlisting}
After this tactic, the propositional variables \lstinline{MinF} and \lstinline{sinM} state that
\lstinline{M ∈ F} and  \lstinline{s ∈ M}, respectively.
Finally, we combine all the elements available by using some of the benefits
of the type-theoretic framework:
\begin{itemize}
\item
  Logical constructs like implications and universal quantifiers behave as
  functions. For instance, the hypothesis \lstinline{hF} (of type
  \lstinline{∀ M, M ∈ F → M ⊑ C}) can be fed with the term \lstinline{M} to
  obtain the implication  \lstinline{hf M} (having type \lstinline{M ∈ F → M ⊑ C}) and 
  the latter can be applied to \lstinline{MinF}
  (a term for the antecedent) to obtain a term \lstinline{hf M MinF} for the consequent \lstinline{M ⊑ C}.
\item
  The conjunction behaves as a Cartesian
  product, where components correspond to each conjunct. Hence the first
  component
  \lstinline{(hf M MinF).1} is a term justifying \lstinline{M ⊆ C} = \lstinline{∀ ⦃s⦄, s ∈ M → s ∈ C}.
\end{itemize}
\lstset{basicstyle=\ttfamily}
By applying the last term obtained to \lstinline{sinM}, we obtain \lstinline{exact}ly
what we were looking for, and the sub-proof ends.
\lstset{basicstyle=\ttfamily\small}
\begin{lstlisting}
    exact (hF M MinF).1 sinM
\end{lstlisting}

For the definition of goodness, we declare our own class consisting of types
supporting a partial order, and we add an otherwise unspecified $g$. A special
type of comment (a \emph{docstring}) describes the concept introduced:
\begin{lstlisting}
/--
A partial order with an *expander* function from subsets to subsets. In main applications, the
expander actually returns a bigger subset.
-/
class OrderExpander (α : Type*) [PartialOrder α] where
  g : Set α → Set α
\end{lstlisting}
Assuming the appropriate structures on $\alpha$ we are finally able to write
down the definition.
\begin{lstlisting}
variable [PartialOrder α] [OrderExpander α]
  
def Good (C : Set α) := IsChain ( · ≤ ·) C ∧ ∀ {S}, S ⊏ C → S ⊏ g S ∧ g S ⊑ C
\end{lstlisting}

A further class concerns partial orders with an “$f$”,
\begin{lstlisting}
class OrderSelector (α : Type*) [PartialOrder α] where
  f : Set α → α
\end{lstlisting}
and a statement that each
type supporting it is an “instance” of \lstinline{OrderExpander} in a canonical way. This is
justified  by presenting 
$C \stackrel{g}{\longmapsto} C \cup \{ f(C) \}$ as the witness.
\begin{lstlisting}
instance [PartialOrder α] [OrderSelector α] : OrderExpander α := ⟨fun C => C ∪ {OrderSelector.f C}⟩
\end{lstlisting}

We skip directly to the statement of the Unbounded Chain Lemma,
\begin{lstlisting}
lemma unbounded_chain [PartialOrder α] [Inhabited α] :
     ∃ C, IsChain ( · ≤ ·) C ∧ ¬ ∃ sb : α, ∀ a ∈ C, a < sb
\end{lstlisting}
which moreover assumes $\alpha$ to be nonempty
(more precisely, that has a designated element) for simplicity,
and the proof of Zorn's Lemma using it:
\begin{lstlisting}
lemma zorn [PartialOrder α] [Inhabited α]
    (ind : ∀ (C : Set α), IsChain ( · ≤ ·) C → ∃ ub, ∀ a ∈ C, a ≤ ub) : ∃ (x : α), IsMaximal x := by
  obtain ⟨C, chain, subd⟩ := unbounded_chain (α := α)
  push_neg at subd
  obtain ⟨ub, hub⟩ := ind C chain
  existsi ub
  intro z hz
  by_contra zneub
  obtain ⟨a, ainC, anltz⟩ := subd z
  exact anltz $ lt_of_le_of_lt (hub a ainC) $ lt_of_le_of_ne' hz zneub
\end{lstlisting}
We comment briefly on some of the tactics employed in this elementary proof. As
before, \lstinline{obtain} decomposes the statement of
\lstinline{unbounded_chain}, and in particular \lstinline{subd} is a term
asserting the truth of $\neg$~\lstinline{∃ sb : α, ∀ a ∈ C, a < sb}. The tactic
\lstinline{push_neg} applies the De Morgan rules transforming it into
\lstinline{∀ (sb : α), ∃ a ∈ C, ¬a < sb}. The obtained upper bound
\lstinline{ub} for the (strictly) unbounded chain is presented as a witness to the existential quantifier of
the conclusion by using the 
\lstinline{existsi} tactic. After introducing variables \lstinline{h} and \lstinline{hz},
the \lstinline{by_contra} tactic starts a proof by contradiction where the new
hypothesis (the negation of the goal appearing immediately before) is stored in
the variable \lstinline{zneub}.

We would like to add a final word on the description of this proof of Zorn's
Lemma as the “leanest” one: Mathlib's version of its formalization \cite{zorn.lean} was
\emph{ported} (translated) from Isabelle/HOL. This new proof was originally
conceived for and written in Lean.

%% file: chain_bounding.bbl
\providecommand{\noopsort}[1]{}
\begin{small}\end{small}